\documentclass[10pt]{article}

\usepackage[a4paper, left=1in,right=1in,top=1in,bottom=1in]{geometry}
\usepackage{latexsym}
\usepackage{amsfonts,amsmath,amssymb,amsthm}
\newtheorem{thm}{Theorem}
\newtheorem{cor}[thm]{Corollary}

\newtheorem{con}[thm]{Conjecture}
\theoremstyle{remark}
\newtheorem{rem}[thm]{Remark}
\newtheorem{defi}[thm]{Definition}
\newtheorem{ex}[thm]{Example}
\newtheorem{prop}{Proposition}

\def\bbb{\mathbb}
\def\cal{\mathcal}

\newcommand{\bs}{\backslash}
\newcommand{\N}{\mathbb{N}}
\newcommand{\Z}{\mathbb{Z}}

\newcommand{\Q}{\mathbb{Q}}

\begin{document}

\author{Chenying Wang\footnote{The research of Chenying Wang was supported
by the Chinese National Science Foundation - Youth grant. Project no.: 11601239.}\\School of Mathematics~$\&$~Statistics\\
Nanjing University of Information Science~$\&$~Technology\\
Nanjing, P. R. China\and Piotr Miska\\Faculty of Mathematics and Computer Science\\Jagiellonian University\\Cracow, Poland\and Istv\'an Mez\H{o}\footnote{The research of Istv\'an Mez\H{o} was supported by the Scientific Research Foundation of Nanjing University of Information Science \& Technology, the Startup Foundation for Introducing Talent of NUIST. Project no.: S8113062001, and the National Natural Science Foundation for China. Grant no. 11501299.}\\School of Mathematics~$\&$~Statistics\\
Nanjing University of Information Science~$\&$~Technology\\
Nanjing, P. R. China}

\date{}

\title{The $r$-derangement numbers}
\maketitle

\begin{abstract}The classical derangement numbers count fixed point-free permutations.
In this paper we study the enumeration problem of generalized derangements, when some
of the elements are restricted to be in distinct cycles in the cycle decomposition.
We find exact formula, combinatorial relations for these numbers as well as analytic
and asymptotic description. Moreover, we study deeper number theoretical properties,
like modularity, $p$-adic valuations, and diophantine problems.
\end{abstract}

\section{Introduction}

The derangement number $D(n)$ denotes the number of fixed point-free
permutations (FPF for short) on $n$ letters. The inclusion-exclusion
principle yields a simple explicit expression for $D(n)$:
\begin{equation}
D(n)=n!\left(1-\frac{1}{1!}+\frac{1}{2!}+\cdots+\frac{(-1)^n}{n!}\right).\label{Dnclosed}
\end{equation}
The recursion
\begin{equation}
D(n)=(n-1)(D(n-1)+D(n-2))\quad(n\ge2)\label{Dnrec}
\end{equation}
with the initial values $D(0)=1$, $D(1)=0$ is easy to reproduce by simple combinatorial arguments.

An interesting fact what follows directly from \eqref{Dnclosed} is that
\begin{equation}
\frac{D(n)}{n!}\to\frac{1}{e}\quad(n\to\infty).\label{Dnconverges}
\end{equation}
The probabilistic interpretation of this fact is that for large $n$ a randomly selected permutation fails to have fixed point with probability $1/e\approx 0.367879441$ (in other words, this is the asymptotic ratio of FPF permutations in the whole). More on the derangement numbers can be found in \cite{Bona}.

This paper is devoted to study a \emph{subclass} of FPF permuations. Let us take the cycle decomposition of such a permutation. The FPF property obviously means that in this permutation any cycle is of length greater than one. What we add to this requirement is the following. We take a permutation on $n+r$ letters and we restrict the first $r$ of these to be in \emph{distinct} cycles. We arrive at the definition of the subject of the paper.

\begin{defi}An FPF permutation on $n+r$ letters will be called \emph{FPF $r$-permutation}
if in its cycle decomposition the first $r$ letters appear to be in \emph{distinct} cycles.
The number of FPF $r$-permutations denote by $D_r(n)$ and call
\emph{$r$-derangement number}. The first $r$ elements, as well as the cycles they are
contained in, will be called \emph{distinguished}.
\end{defi}

This definition was motivated by the extensive study of the so-called $r$-Stirling
numbers of the first kind \cite{Broder} which count permutations with a fixed number
of cycles where the same restriction on the first distinguished elements is added.
Without this restriction we get the classical Stirling numbers \cite{GKP}.

Some recent (and not so recent) papers are studying this restriction with respect
to other combinatorial objects, like set partitions \cite{Broder}, ordered
lists \cite{Nyul}, permutation statistics \cite{Shattuck}.

It follows from the definition that $n$ must be greater than or equal
to $r$, i.e., $D_r(n)=0$ if $n<r$ and it is equally easy to see that
\begin{equation}
D_1(n)=D(n+1),\quad D_r(r)=r!\;(r\ge1),\quad\mbox{and}\quad D_r(r+1)=r(r+1)!\;(r\ge2).\label{specvals}
\end{equation}

These are the initial values for the below basic recursion of the $r$-derangement numbers.

\begin{thm}\label{prop_recur}For all $n>2$ and $r>0$ we have that
\begin{equation}
D_r(n)=rD_{r-1}(n-1)+(n-1)D_r(n-2)+(n+r-1)D_r(n-1).\label{Dnrrecur}
\end{equation}
\end{thm}

\begin{proof} If we would like to construct an FPF $r$-permutation on $n+r$ elements recursively, we can start with a similar permutation on $n+r-1$ elements. Adding the last element $n+r$ to such a permutation we have two main cases:

\begin{enumerate}
\item The new element is in a transposition (i.e., a two-length cycle). In this case we have two sub-cases.\begin{enumerate}
\item The new element shares its cycle with a distinguished element. This offers $r$ cases for this cycle. Also, the permutation we start with already must be an FPF $(r-1)$-partition on $n+r-2$ elements. There are $D_{r-1}(n-1)$ such permutations. The first term on the right now comes.

\item The new element shares its cycle with one non-distinguished element from that of $n-1$. The rest of the permutation is an FPF $r$-permutation on $n+r-2$ elements. This explains the second term in \eqref{Dnrrecur}.
\end{enumerate}

\item The new element is in a cycle longer than 2. Then this element is inserted somewhere between two elements in the permutation or at the end. Since we have $n+r-1$ elements, we have $n+r-1$ different places to insert. The number of initial permutations is $D_r(n-1)$. This case is counted by the third term.
\end{enumerate}
\end{proof}

An additional complexity of the recursion, compared to the particularly simple \eqref{Dnrec}, comes from the fact that it uses not only the previous two elements of the sequence but an element from the sequence with index $r-1$.

The first members of the sequence starting from $D_2(2)$ are
\[2,\, 12,\, 84,\, 640,\, 5\,430,\, 50\,988,\, 526\,568,\, 5\,940\,576,\, 72\,755\,370,\, 961\,839\,340,\, 13\,656\,650\,172, \dots,\]
while the first members of $D_3(n)$ starting from $n=3$ are
\[6,\, 72,\, 780,\, 8\,520,\, 97\,650,\, 1\,189\,104,\, 15\,441\,048,\, 213\,816\,240,\, 3\,152\,287\,710,\, 49\,369\,524\,600,\dots.\]

\section{Fundamental properties of the $r$-derangement numbers}

\subsection{Exponential generating function of $D_r(n)$}

First, we give the exponential generating function of the sequence of $r$-derangements numbers and then deduce some combinatorial relations.

\begin{thm}
For any $r\in\N$ for the exponential generating function of the sequence of $r$-derangements numbers we have that
\begin{equation*}
F_r(x):=\sum_{n=0}^{+\infty} \frac{D_r(n)}{n!}x^n=\frac{x^re^{-x}}{(1-x)^{r+1}}.
\end{equation*}
\end{thm}

\begin{proof}
We will prove the statement of the theorem by induction on $r$.

If $r=0$ then $D_r(n)=D(n)$, $n\in\N$, is a classical derangement number and one can easily show \cite[p. 106, Example 3.56]{Bona} that the exponential generating function of the derangement numbers is $F_0(x)=\frac{e^{-x}}{1-x}$.

If $r=1$ then $D_r(n)=D(n+1)$ for any $n\in\N$, hence
\begin{equation*}
\begin{split}
F_1(x) & =\sum_{n=0}^{+\infty} \frac{D(n+1)}{n!}x^n=\sum_{n=0}^{+\infty} \frac{D(n+1)}{(n+1)!}\cdot (n+1)x^n =\sum_{n=1}^{+\infty} \frac{D(n)}{n!}\cdot nx^{n-1}=F'_0(x)=\frac{xe^{-x}}{(1-x)^2}.
\end{split}
\end{equation*}

Let us assume now that $r\geq 2$ and $F_{r-1}(x)=\frac{x^{r-1}e^{-x}}{(1-x)^r}$. We apply our recursion \eqref{Dnrrecur} and the fact that $D_r(n)=0$ for $n<r$ to obtain the following
\[F_r(x) =\sum_{n=r}^{+\infty} \frac{D_r(n)}{n!}x^n
=\sum_{n=r}^{+\infty}\left(\frac{r}{n}\cdot\frac{D_{r-1}(n-1)}{(n-1)!}
                           +\frac{1}{n}\cdot\frac{D_r(n-2)}{(n-2)!}
                           +\frac{n+r-1}{n}\cdot\frac{D_r(n-1)}{(n-1)!}\right)x^n.\]
After differentiation this turns to be
\begin{equation*}
\begin{split}
F'_r(x) & = r\cdot \sum_{n=r}^{+\infty} \frac{D_{r-1}(n-1)}{(n-1)!}x^{n-1} + x\cdot\sum_{n=r}^{+\infty} \frac{D_{r}(n-2)}{(n-2)!}x^{n-2}+ \\ & + (r-1)\cdot\sum_{n=r}^{+\infty} \frac{D_{r}(n-1)}{(n-1)!}x^{n-1} + \sum_{n=r}^{+\infty} \frac{D_{r}(n-1)}{(n-1)!}\cdot nx^{n-1} \\
& = rF_{r-1}(x) + (x+r-1)F_r(x) + (xF_r(x))' = \frac{rx^{r-1}e^{-x}}{(1-x)^r} + (x+r)F_r(x) + xF'_r(x),
\end{split}
\end{equation*}
by the induction hypothesis. We thus obtain the following nonhomogeneous linear differential equation
\begin{equation}\label{eq1}
(1-x)F'_r(x) = \frac{rx^{r-1}e^{-x}}{(1-x)^r} + (x+r)F_r(x)
\end{equation}
with initial condition $F_r(0)=D_r(0)=0$. The solution of the equation \eqref{eq1} is unique and it can be checked easily that it is $\frac{x^re^{-x}}{(1-x)^{r+1}}$, indeed.
\end{proof}

\subsection{Combinatorial identities}

Knowing the exponential generating function of the $r$-derangement number sequences allows us to deduce some identities. These will be proven by combinatorial arguments, too.

\begin{thm}
Let $r\in\N_+$ and $s\in\{1,...,r\}$. Then for each $n\geq s$ we have
\begin{equation}\label{eq2}
D_r(n)=\sum_{j=s}^n {j-1\choose s-1}\frac{n!}{(n-j)!}D_{r-s}(n-j).
\end{equation}
In particular,
\begin{equation}\label{eq3}
D_r(n)=\sum_{j=r}^n {j-1\choose r-1}\frac{n!}{(n-j)!}D(n-j),\mbox{ } n\geq r.
\end{equation}
Additionally, we have a closed formula for $r$-derangements numbers:
\begin{equation}\label{eq4}
D_r(n)=\sum_{j=r}^n {j\choose r}\frac{n!}{(n-j)!}(-1)^{n-j},\mbox{ } n\geq r.
\end{equation}
\end{thm}

\begin{proof}
Let us expand the function $F_r$ as follows:
\begin{equation*}
\begin{split}
F_r(x) & =\frac{x^re^{-x}}{(1-x)^{r+1}}=\left(\frac{x}{1-x}\right)^s\cdot F_{r-s}(x)=\left(\sum_{j=1}^{+\infty} x^j\right)^s\cdot\left(\sum_{k=0}^{+\infty} \frac{D_{r-s}(k)}{k!}x^k\right) \\
& =\left(\sum_{j=s}^{+\infty} {j-1\choose s-1}x^j\right)\cdot\left(\sum_{k=0}^{+\infty} \frac{D_{r-s}(k)}{k!}x^k\right)=\sum_{n=s}^{+\infty} \left(\sum_{j=s}^{+\infty} {j-1\choose s-1}\frac{D_{r-s}(n-j)}{(n-j)!}\right)x^n.
\end{split}
\end{equation*}
By comparing the coefficients we conclude the equality
\begin{equation*}
\frac{D_r(n)}{n!}=\sum_{j=s}^n {j-1\choose s-1}\frac{D_{r-s}(n-j)}{(n-j)!}
\end{equation*}
for $n\geq s$. This establishes identities (\ref{eq2}) and (\ref{eq3}).

For the proof of the closed formula for $D_r(n)$ we use \eqref{Dnclosed}:
\begin{equation*}
\begin{split}
D_r(n) & =\sum_{j=r}^n {j-1\choose r-1}\frac{n!}{(n-j)!}D(n-j)=n!\cdot\sum_{j=r}^n {j-1\choose r-1}\left(\sum_{k=0}^{n-j} \frac{(-1)^k}{k!}\right) \\
& =n!\cdot\sum_{k=0}^{n-r} \frac{(-1)^k}{k!}\left(\sum_{j=r}^{n-k} {j-1\choose r-1}\right)=n!\cdot\sum_{k=0}^{n-r} {n-k\choose r}\frac{(-1)^k}{k!} =\sum_{j=r}^n {j\choose r}\frac{n!}{(n-j)!}(-1)^{n-j}.
\end{split}
\end{equation*}
\end{proof}

Identity \eqref{eq4} can be proven in another way. Namely, we can expand the function $F_r(x)$ as follows:
\begin{equation*}
\begin{split}
F_r(x) & =\frac{1}{x}\cdot\left(\frac{x}{1-x}\right)^{r+1}\cdot e^{-x}=\frac{1}{x}\cdot\left(\sum_{j=r+1}^{+\infty} {j-1\choose r}x^j\right)\cdot\left(\sum_{k=0}^{+\infty} \frac{(-1)^k}{k!}x^k\right) \\
& =\left(\sum_{j=r}^{+\infty} {j\choose r}x^j\right)\cdot\left(\sum_{k=0}^{+\infty} \frac{(-1)^k}{k!}x^k\right)=\sum_{n=r}^{+\infty} \left(\sum_{j=r}^{n} {j\choose r}\frac{(-1)^{n-j}}{(n-j)!}\right)x^n
\end{split}
\end{equation*}
and compare the corresponding coefficients on both sides.

\begin{rem}
Moreover, one can prove (\ref{eq2}) by using the following combinatorial argument, too. Let us fix $r,s\in\N_+$ with $s\leq r$. In order to construct an $r$-derangement we choose a number $j$ to be a number of non-distinguished elements which will be contained in $s$ first distinguished cycles. Certainly, $j\geq s$. Next we choose these $j$ non-distinguished elements. The $s$ cycles built by them have lengths $1+i_1$, ..., $1+i_s$ respectively, where $i_1, ..., i_s>0$ and $i_1+...+i_s=j$. We choose the $s$-tuple $(i_1, ..., i_s)$ in $j-1\choose s-1$ ways. Next we choose the first non-distingushed element after $1$ in the first distinguished cycle in $n$ ways, the second element after $1$ in $n-1$ ways, and at the end, the $i_1$-th element after $1$ in $n-i_1+1$ ways. Moreover, we choose the first non-distinguished element after $2$ in the second distinguished cycle in $n-i_1$ ways, the second element after $2$ in $n-i_1-1$, and so on. Hence, in total, we choose $j$ non-distinguished elements in distinguished cycles in $\frac{n!}{(n-j)!}$ ways. Finally, the remaining $n-j$ non-distinguished elements and $r-s$ distinguished elements create an $(r-s)$-derangement. We can choose this $(r-s)$-derangement in $D_{r-s}(n-j)$ ways. We then conclude that an $r$-derangement can be chosen in $\sum_{j=s}^{n} {j-1\choose s-1}\frac{n!}{(n-j)!}D_{r-s}(n-j)$ ways.
\end{rem}

\subsection{Asymptotics}

By using the above results we are going to establish asymptotic estimates for $D_r(n)$. Later we will use probabilistic arguments and Lah numbers, so that we will provide three independent proofs.

\begin{thm}\label{t2}
If $r\in\N_+$ and $n\geq r$ then $\left|D_r(n)-\frac{n!}{e}{n-1\choose r}\right|<2n!{n-1\choose r-1}$. In particular, for each $r\in\N_+$ the sequence $\left(\frac{D_r(n)}{(n+r)!}\right)_{n\in\N}$ is convergent to $\frac{1}{r!\cdot e}$.
\end{thm}

\begin{proof}
One can easily show that $D(n)$ is the best integer approximation of the number $\frac{n!}{e}$ for $n\in\N_+$. Using identity (\ref{eq3}) we obtain
\begin{equation*}
\begin{split}
& D_r(n)=\sum_{j=r}^n {j-1\choose r-1}\frac{n!}{(n-j)!}D(n-j) = \sum_{j=r}^{n-1} {j-1\choose r-1}\frac{n!}{(n-j)!}\left(\frac{(n-j)!}{e}+\xi_j\right)+{n-1\choose r-1}n! \\ = & \sum_{j=r}^{n-1} {j-1\choose r-1}\frac{n!}{e}+\sum_{j=r}^{n-1} {j-1\choose r-1}\frac{n!}{(n-j)!}\xi_j+{n-1\choose r-1}n! \\
= & {n-1\choose r}\frac{n!}{e}+\frac{n!}{(r-1)!}\sum_{j=r}^{n-1} \frac{(j-1)\cdot ...\cdot (j-r+1)}{(n-j)!}\xi_j+{n-1\choose r-1}n!,
\end{split}
\end{equation*}
where $|\xi_j|<\frac{1}{2}$ for $j\in\{r,...,n-1\}$. Hence
\begin{equation}\label{eq5}
\begin{split}
& \left|D_r(n)-\frac{n!}{e}{n-1\choose r}\right|=\left|\frac{n!}{(r-1)!}\sum_{j=r}^{n-1} \frac{(j-1)\cdot ...\cdot (j-r+1)}{(n-j)!}\xi_j+{n-1\choose r-1}n!\right|< \\
< & \frac{n!}{(r-1)!}\sum_{j=r}^{n-1} \frac{(j-1)\cdot ...\cdot (j-r+1)}{2(n-j)!}+{n-1\choose r-1}n!\leq \\
\leq & \frac{n!}{(r-1)!}\sum_{j=r}^{n-1} \frac{(n-1)\cdot ...\cdot (n-r+1)}{2^{n-j}}+{n-1\choose r-1}n!< \\
< & \frac{n!}{(r-1)!}(n-1)\cdot ...\cdot (n-r+1)+{n-1\choose r-1}n!=2{n-1\choose r-1}n!,
\end{split}
\end{equation}
which proves the first part of the statement. In order to prove the second part,
it suffices to divide (\ref{eq5}) by $(n+r)!$ and let $n$ tend to $+\infty$.
\end{proof}

By the saddle point method \cite{Wilf} we can find another way to get the above asymptotic estimation. This approach actually provides an even better approximation.

\begin{thm}We have that
\begin{equation}
\frac{D_r(n)}{n!}=\frac{1}{e}\sum_{k=0}^{r}A(r,k)\binom{n+r-k}{n}+O(\varepsilon^n).\label{est}
\end{equation}
where
\[A(r,k)=\sum_{i=0}^r\frac{(-1)^i}{(k-i)!}\binom{r}{i},\quad((k-i)!=0\;\text{when}\; i>k)\]
and $\varepsilon>0$ is an arbitrarily small real number.
\end{thm}

\begin{proof} First denote $\frac{x^re^x}{(1-x)^{r+1}}$ by $F_r(x)$.
Then we note that the Laurent expansion of $x^re^{-x}$  around $x=1$ is
\[\frac{1}{e}\sum_{k=0}^{r}A(r,k)(1-x)^k+O\left((1-x)^{r+1}\right),\]
so we can find the principal part of $F_r(x)$ around its unique singularity $x=1$.
This equals
\[PP(F_r(x),1)=\frac{1}{e}\frac{1}{(1-x)^{r+1}}\sum_{k=0}^{r}A(r,k)(1-x)^k.\]
The saddle point method says that the $n$th coefficient of $F_r(x)$ equals
the $n$th coefficient of the principal part plus the contribution from the
regular part $F_r(x)-PP(F_r(x),1)$. The regular part is an entire function,
since $F_r(x)$ has no other singularity other than $x=1$. Therefore the
contribution of the regular part is $O(\varepsilon^n)$ for an arbitrary
$\varepsilon>0$. So we have that
\[\frac{D_r(n)}{n!}\sim[x^n]\frac{1}{e}\frac{1}{(1-x)^{r+1}}\sum_{k=0}^{r}A(r,k)(1-x)^k+O(\varepsilon^n).\]
The coefficients of $x^n$ on the right hand side can be found easily, thus we finally arrive at the statement of the theorem.
\end{proof}

\begin{rem} The above theorem is a refinement of $\frac{D_r(n)}{(n+r)!}\to\frac{1}{r!e}$. To see this, we write out the particular case $r=2$:
\[\frac{(-1)^n}{e}\sum_{k=0}^{2}A(2,k)\binom{k-1-2}{n}=\frac{1}{2e}(n^2+n-1),\]
hence
\[\frac{D_2(n)}{(n+2)!}=\frac{1}{2e}\frac{n^2+n-1}{(n+1)(n+2)}+O(\varepsilon^n)\to\frac{1}{2e},\]
indeed.

Similarly for $r=3$:
\[\frac{(-1)^n}{e}\sum_{k=0}^{3}A(3,k)\binom{k-1-3}{n}=\frac{1}{2e}(n^3-4n+1),\]
and then
\[\frac{D_3(n)}{(n+3)!}=\frac{1}{6e}\frac{n^3-4n+1}{(n+1)(n+2)(n+3)}+O(\varepsilon^n)\to\frac{1}{6e}.\]
We note that this approximation is rather close even for small values of $n$. For example,
\[\frac{D_4(8)}{(8+4)!}=0.00351080246\dots,\]
while from the approximation \eqref{est} we get the estimation
\[\frac{D_4(8)}{(8+4)!}\approx0.00351080232\dots.\]
Nine digits already agree for $n$ as small  as $n=8$.
\end{rem}

\section{A connection with the Lah numbers}

\subsection{Probabilistic approach}

The Lah numbers $L(n,k)$ are defined by
\[L(n,k)=\frac{n!}{k!}\binom{n-1}{k-1}.\]
These numbers count the partitions of $n$ elements into $k$ blocks such that the order of the elements in the individual blocks count, but the order of the blocks is not taken into account. Such partitions are often called ordered lists.

We now show that these numbers are connected to the $r$-derangements.

\begin{thm}\label{Lahprop}Let $r\in\N$ and $n\in\N_+$ such that $n\geq r$. Then
\[(r+1)!L(n,r+1)=\sum_{k=1}^n\binom{n}{k}kD_r(n-k).\]
\end{thm}

\begin{proof} The proof uses a probabilistic argument. Let $\mathcal{P}_{n,r}$ be the set of permutations on $n+r$ elements such that the first $r$ elements are not fixed points and they are in different cycles. It is easy to see that
\begin{equation}
P_{n,r}:=|\mathcal{P}_{n,r}|=\frac{n!}{r!}n(n-1)\cdots(n-r+1)\quad(n\ge r).\label{Pnr}
\end{equation}

We define the probability distribution
\[p_k^{(n)}=\frac{\binom{n}{k}D_r(n-k)}{P_{n,r}}\quad(k=0,1,\dots,n).\]
Here $p_k^{(n)}$ is the probability of the event that if we take a
permutation from $\mathcal{P}_{n,r}$ randomly and uniformly, then
this permutation contains $k$ fixed points. Let $X$ be the number of
fixed points in such a permutation. Then the expectation of $X$ is
\[\frac{1}{P_{n,r}}\sum_{k=1}^n\binom{n}{k}kD_r(n-k).\]

On the other hand, this expectation can be determined as follows.
Taking a point from $\{r+1,r+2,\dots,r+n\}$ ($1,2,\dots,r$
cannot be fixed points, by definition), the probability that it
is a fixed point equals $\frac{P_{n-1,r}}{P_{n,r}}$. Summing
over all possible points and using the linearity of the expectation,
we get that it equals
\[n\frac{P_{n-1,r}}{P_{n,r}}=\frac{n-r}{n},\]
by \eqref{Pnr}. In the particular case when $r=0$ we get back the classical fact that the expected number of fixed points in a randomly chosen permutation is one.

At this point we have that
\[\sum_{k=1}^n\binom{n}{k}kD_r(n-k)=P_{n,r}\frac{n-r}{n}.\]
This, by using \eqref{Pnr} can be rerewritten by using the Lah numbers.
\end{proof}

%\begin{cor}We have that
%\[\lim_{n\to\infty}\frac{D_r(n)}{(n+r)!}=\frac{1}{r!e}.\]
%\end{cor}
%
%\begin{proof}Let us define the probability distribution $p_k^{(n)}$ as above:
%\[p_k^{(n)}=\frac{\binom{n}{k}D_r(n-k)}{P_{n,r}}.\]
%By \eqref{Pnr} we have that
%\[p_k^{(n)}=\frac{r!}{k!}\frac{D_r(n-k)}{(n-k)!n(n-1)\cdots(n-r+1)}\quad(k=0,1,\dots,n).\]
%Moreover, for any fixed $k$ the
%\[\frac{D_r(n-k)}{(n-k)!n(n-1)\cdots(n-r+1)}\sim\frac{D_r(n)}{(n+r)!}\quad(n\to\infty),\]
%asymptotics holds regardless of the actual value of $k$. Now summing $p_k^{(n)}$ over $k$, we have that
%\[1=\sum_{k=0}^\infty p_k^{(n)}\sim r!\frac{D_r(n)}{(n+r)!}\sum_{k=0}^\infty\frac{1}{k!}\sim r!e\frac{D_r(n)}{(n+r)!},\]
%from where we conclude that $\lim_{n\to\infty}\frac{D_r(n)}{(n+r)!}=\frac{1}{r!e}$, indeed.
%\end{proof}

\subsection{Combinatorial approach}

Now we provide another proof of Theorem \ref{Lahprop} which, in turn, uses a combinatorial argument.

First, let us assume that $r=0$. Then the equality of Theorem \ref{Lahprop} takes the form $n!=\sum_{k=1}^n {n\choose k}kD(n-k)$ and it can be justified by the fact that the expected number of fixed points of a random permutation is equal to $1$ (each permutation of a set with $n$ elements can be treated as a derangement of a set with $n-k$ elements, where $k$ is the number of fixed points of this permutation).

Now, consider the case $r>0$. $L(n,r+1)$ is the number of partitions of $n$ elements into $r+1$ sequences, where the order of sequences is not mentioned. Hence $(r+1)!L(n,r+1)$ is the number of partitions of $n$ elements into $r+1$ sequences, where their order is taken into account. On the other hand, we can determine each partition in the following way. At first, we fix a number $j\in\{r,...,n\}$ to be a number of elements creating first $r$ lists. Next, we choose these $j$ elements and set them in some sequence (which can be done in $\frac{n!}{(n-j)!}$ ways). Then we split this sequence into $r$ lists by picking $r-1$ elements from $j-1$ (in ${j-1\choose r-1}$ ways) being first elements of the second, third, ... and $r$-th list (the first element in the beginning sequence becomes the first element of the first list). Finally, we build the $r+1$-st list from the remaining $n-j$ elements (in $(n-j)!$ ways).
%On the other hand, each partition of $n$ elements into $r+1$ enumerated lists can be determined by choice of $j$ elements ($j\geq r$) creating first $r$ lists and setting these elements in some sequence (which can be done in $\frac{n!}{(n-j)!}$ ways), splitting this sequence into $r$ lists by picking $r-1$ elements from $j-1$ (in $j-1\choose r-1$ ways) being first elements in second, third, ... and $r$-th list (the first element in the beginning sequence becomes the first element in the first list) and finally building the $(r+1)$-st list from the remaining $n-j$ elements (in $(n-j)!$ ways).
We thus obtain the equality
\[(r+1)!L(n,r+1)=\sum_{j=r}^n {j-1\choose r-1}\frac{n!}{(n-j)!}(n-j)!.\]
Now we are making use of the simple fact that
\[(n-j)!=\sum_{k=0}^{n-j} {n-j\choose k}D(n-j-k)\]
together with \eqref{eq3} in order to obtain the following chain of equalities:
\begin{equation*}
\begin{split}
& \sum_{j=r}^n {j-1\choose r-1}\frac{n!}{(n-j)!}(n-j)! = \sum_{j=r}^n {j-1\choose r-1}\frac{n!}{(n-j)!}\sum_{k=0}^{n-j} \frac{(n-j)!}{k!(n-j-k)!}kD(n-j-k) \\
= & \sum_{j=r}^n\sum_{k=0}^{n-j} {j-1\choose r-1}\frac{n!}{(n-j)!} \frac{(n-j)!(n-k)!}{n!(n-j-k)!}{n\choose k}kD(n-j-k) \\
= & \sum_{j=r}^n\sum_{k=0}^{n-j} {j-1\choose r-1} \frac{(n-k)!}{(n-j-k)!}{n\choose k}kD(n-j-k) \\
= & \sum_{k=0}^{n-r} {n\choose k}k \sum_{j=r}^{n-k} {j-1\choose r-1} \frac{(n-k)!}{(n-j-k)!}D(n-j-k) = \sum_{k=0}^n {n\choose k}kD_r(n-k).
\end{split}
\end{equation*}
So the second proof of Theorem \ref{Lahprop} is complete.

\section{Polynomials related to $r$-derangements numbers}\label{polyrel}
Let us fix $n\in\N$. We investigate how the values of the sequence
$(D_r(n+r))_{r\in\N}$ can be expressed by values of some polynomial.
Let us use the exact formula for $D_r(n+r)$.
\begin{equation*}
\begin{split}
& D_r(n+r)=\sum_{j=r}^{n+r} {j\choose r}\frac{(n+r)!}{(n+r-j)!}(-1)^{n+r-j}=\sum_{j=0}^{n} {j+r\choose j}\frac{(n+r)!}{(n-j)!}(-1)^{n-j} \\
& =(n+r)_r\sum_{j=0}^{n} \frac{(j+r)_j}{j!}\cdot\frac{n!}{(n-j)!}(-1)^{n-j}=(n+r)_r\sum_{j=0}^{n} (j+r)_j{n\choose j}(-1)^{n-j},
\end{split}
\end{equation*}
where we use the Pochhammer symbol $(n)_r=n\cdot ...\cdot (n-r+1)$
for $r\in\N_+$ and $(n)_0=1$. We define the polynomial
$P_n(X)=\sum_{j=0}^{n} (j+X)_j{n\choose j}(-1)^{n-j}\in\Z[X]$.
Then we can write $D_r(n+r)=(n+r)_rP_n(r)$ for $r\in\N$.

Using the identity $\binom{n}{j}=\binom{n-1}{j}+\binom{n-1}{j-1}$ it is easy
to see that the following holds:
\begin{prop}
We have a recurrence relation for polynomials $P_n$:
\begin{equation}\label{rec}
P_0(X)=1, P_n(X)=(X+1)P_{n-1}(X+1)-P_{n-1}(X), n>0.
\end{equation}
\end{prop}

%\begin{proof}
%Clearly $P_0=1$. For $n>0$ we have
%\begin{equation*}
%\begin{split}
%& P_n(X)=\sum_{j=0}^{n} (j+X)_j{n\choose j}(-1)^{n-j} \\
%& =(-1)^n+\sum_{j=1}^{n-1} (j+X)_j\left({n-1\choose j-1}+{n-1\choose j}\right)(-1)^{n-j}+(n+X)_n \\
%& =(-1)^n+\sum_{j=1}^{n-1} (j+X)_j{n-1\choose j}(-1)^{n-j}+\sum_{j=1}^{n} (j+X)_j{n-1\choose j-1}(-1)^{n-j} \\
%& =-P_{n-1}(X)+(X+1)\sum_{j=1}^{n} (j+X)_{j-1}{n-1\choose j-1}(-1)^{n-j} \\
%& =-P_{n-1}(X)+(X+1)P_{n-1}(X+1).
%\end{split}
%\end{equation*}
%\end{proof}

Using the recurrence relation for polynomials $P_n(X)$ we can obtain one
more identity for $r$-derangements numbers. This identity allows us to
write $D_{r+1}(n)$ in terms of $D_r(n)$ and $D_r(n-1)$ and is an easy exercise.
\begin{cor}
For any $r\in\N$ and $n\in\N_+$ we have
\begin{equation}\label{iden}
D_{r+1}(n)=\frac{n-r}{r+1}D_r(n)+\frac{n}{r+1}D_r(n-1).
\end{equation}
\end{cor}

%\begin{proof}
%If we insert $r$ in the place of $X$ in the recurrence
%$$P_n(X)=(X+1)P_{n-1}(X+1)-P_{n-1}(X)$$
%and multiply by $(n+r)_r$ then we get
%$$D_r(n+r)=\frac{r+1}{n}D_{r+1}(n+r)-\frac{n+r}{n}D_r(n+r-1).$$
%Hence we easily obtain
%$$D_{r+1}(n+r)=\frac{n}{r+1}D_r(n+r)+\frac{n+r}{r+1}D_r(n+r-1).$$
%After substitution $n$ in place of $n-r$ we get the statement.
%\end{proof}

Let us see that for $r=0$ identity \eqref{iden} becomes the well-known
recurrence for numbers of classical derangements $D(n+1)=n(D(n)+D(n-1))$, $n\in\N_+$.

Via a straightforward induction argument based on the parity of $n$,
recurrence \eqref{rec} allows us to factorize reductions of polynomials $P_n$ modulo $2$.

\begin{prop}
For each $n\in\N$ there holds
\begin{equation*}
P_n(X)\pmod{2}=
\begin{cases}
(X^2+X+1)^{\frac{n}{2}}, & \mbox{ if } 2\mid n, \\
X(X^2+X+1)^{\frac{n-1}{2}}, & \mbox{ if } 2\nmid n.
\end{cases}
\end{equation*}
\end{prop}

%\begin{proof}
%Obviously, $P_0\equiv 1\pmod{2}$. Let us assume that $P_{2m}\equiv (X^2+X+1)^m\pmod{2}$ for some $m\in\N$. Then
%\begin{equation*}
%\begin{split}
%& P_{2m+1}(X)\equiv (X+1)P_{2m}(X+1)+P_{2m}(X) \\
%& \equiv (X+1)(X^2+1+X+1+1)^m+(X^2+X+1)^m \\
%& \equiv (X+1)(X^2+X+1)^m+(X^2+X+1)^m \\
%& \equiv (X+2)(X^2+X+1)^m\equiv X(X^2+X+1)^m\pmod{2}
%\end{split}
%\end{equation*}
%and
%\begin{equation*}
%\begin{split}
%& P_{2m+2}(X)\equiv (X+1)P_{2m+1}(X+1)+P_{2m+1}(X) \\
%& \equiv (X+1)^2(X^2+1+X+1+1)^m+X(X^2+X+1)^m \\
%& \equiv (X^2+1)(X^2+X+1)^m+X(X^2+X+1)^m \\
%& =(X^2+X+2)(X^2+X+1)^m=(X^2+X+1)^{m+1}\pmod{2}
%\end{split}
%\end{equation*}
%\end{proof}

Let us note that coefficients of powers of the trinomial $X^2+X+1$ create well known integer sequences which have combinatorial interpretations. For example, coefficients ${n\choose 0}_2$ of $X^n$ in the expansion of $(X^2+X+1)^n$ are numbers of planar paths from the point $(0,0)$ to the point $(n,0)$, where the only possible moves are $(1,0)$, $(1,1)$ and $(1,-1)$. Another interesting example is the sequence $\left({n\choose 1}_2\right)_{n\in\N}$ of coefficients of $X^{n+1}$ (or equivalently of $X^{n-1}$) in the expansion of $(X^2+X+1)^n$. For $n\in\N_+$ there holds ${n\choose 1}_2=nM_{n-1}$, where $M_n$ is the Motzkin number, which counts all the paths from the point $(0,0)$ to the point $(n,0)$ which do not descend below the $x$-axis and the only possible moves are $(1,0)$, $(1,1)$ and $(1,-1)$. Motzkin number is also the number of ways of drawing any number of nonintersecting chords joining $n$ (labeled) points on a circle (see i.e. \cite{OEIS} and \cite{Weis}).

\section{Periodicity of remainders, prime divisors and $p$-adic valuations}

After studying the analytical and combinatorial properties of the $r$-derangement numbers we turn to number theoretical properties. We prove several modularity results, and study some diophantine equations involving $r$-derangements. Among others, we are going to prove that $D_r(n)$ is a multiple of a factorial number only in finitely many cases.

\subsection{Periodicity}

Let $r\in\N_+$ be fixed. First of all, let us note that $r!\mid D_r(n)$ for all $r\in\N_+$ and $n\in\N$ because if we permute distinguished elements in some $r$-derangement then we obtain another $r$-derangement. Hence, if $d\in\N_+$ is a divisor of $r!$, then the sequence of remainders $(D_r(n)\pmod{d})_{n\in\N}$ is constant and equal to 0. We can prove much more with some additional effort.

\begin{thm}\label{thm_period}
For each $r,d\in\N_+$, if $n_1,n_2\in\N$ and $n_1\equiv n_2\pmod{d}$ then
\begin{equation*}
(-1)^{n_1}D_r(n_1)\equiv (-1)^{n_2}D_r(n_2)\pmod{d}.
\end{equation*}
In particular, the sequence $(D_r(n)\pmod{d})_{n\in\N}$ is periodic of period
\begin{itemize}
\item $d$, if $2\mid d$,
\item $2d$, if $2\nmid d$.
\end{itemize}
\end{thm}

\begin{proof}
We may assume that $d\nmid r!$. For the proof of periodicty of the sequence $(D_r(n)\pmod{d})_{n\in\N}$ we are going to show that
\begin{equation}\label{eq7}
D_r(n)\equiv (-1)^n\sum_{j=r}^{d-1} (-1)^j {j\choose r}(n)_j\pmod{d}
\end{equation}
for any $n\in\N$. In the case when $n\geq r$ we apply (\ref{eq4}):
\begin{equation*}
\begin{split}
D_r(n) & =\sum_{j=r}^n {j\choose r}\frac{n!}{(n-j)!}(-1)^{n-j}= (-1)^n\sum_{j=r}^{n} (-1)^j {j\choose r}(n)_j\\& \equiv (-1)^n\sum_{j=r}^{d-1} (-1)^j {j\choose r}(n)_j\pmod{d}.
\end{split}
\end{equation*}
Assuming that $n\geq d$, making reduction modulo $d$ we can skip
the summands from $d$th to $n$th because if $d\leq j\leq n$ then
among the (at least $d$) numbers $n-j+1$, $n-j+2$, $...$, $n$ there
is a multiple of $d$ and thus $d\mid (n)_j$. If $r<n<d$ then $(n)_j=0$
for $n+1\leq j\leq d$.

Now we consider the case when $n<r$. Then
obviously $D_r(n)=0$ and $(n)_j=0$ for $r\leq j\leq d$, which establishes
congruence (\ref{eq7}) for $n<r$.

Now define $f_{r,d}(X)=\sum_{j=r}^{d-1} (-1)^j {j\choose r}(X)_j\in\Z[X]$. Then congruence (\ref{eq7}) takes the form
\begin{equation*}
D_r(n)\equiv (-1)^nf_{r,d}(n)\pmod{d},\mbox{ } n\in\N,
\end{equation*}
from where we obtain Theorem \ref{thm_period}.
\end{proof}

We can strengthen the divisibility $r!\mid D_r(n)$ for all $r\in\N_+$ and $n\in\N$. Namely, by formula (\ref{eq4}) we obtain easily that $\frac{n!}{(n-r)!}\mid D_r(n)$, where $r\in\N_+$ and $n\geq r$. Indeed,
\begin{equation*}
D_r(n)=\sum_{j=r}^n {j\choose r}\frac{n!}{(n-j)!}(-1)^{n-j}=\frac{n!}{(n-r)!}\sum_{j=r}^n {j\choose r}\frac{(n-r)!}{(n-j)!}(-1)^{n-j}=(n)_r\sum_{j=r}^n {j\choose r}(n-r)_{j-r}(-1)^{n-j}.
\end{equation*}

Let us define
\[C_r(n):=\frac{1}{(n)_r}D_r(n)=\sum_{j=r}^n {j\choose r}(n-r)_{j-r}(-1)^{n-j}\]
for $r\in\N_+$ and $n\geq r$. Note that $C_1(n)=\frac{D_1(n)}{n}=\frac{D(n+1)}{n}$ for $n\in\N_+$. Arithmetic properties of numbers $\frac{D(n+1)}{n}$ were studied in \cite{Mi}. Now we will study properties of numbers $C_r(n)$, $n\geq r$, for arbitrary $r\in\N_+$.

Analogously as for congruence (\ref{eq7}), we are able to obtain the following one:
\begin{equation}\label{eq8}
C_r(n)=\sum_{j=r}^{n} {j\choose r}(-1)^{n-j}(n-r)_{j-r}\equiv (-1)^n\sum_{j=r}^{r+d-1} {j\choose r}(-1)^j(n-r)_{j-r}\pmod{d}.
\end{equation}
Let us define
\begin{equation*}
\hat{f}_{r,d}(X):=\sum_{j=r}^{r+d-1} {j\choose r}(-1)^j(X-r)_{j-r}=\frac{f_{r,r+d}(X)}{X(X-1)...(X-r+1)}\in\Z[X].
\end{equation*}
Then congruence (\ref{eq8}) takes the form
\begin{equation}\label{eq9}
C_r(n)\equiv (-1)^n\hat{f}_{r,d}(n)\pmod{d}, \mbox{ } n\geq r.
\end{equation}
We thus obtain analogous result on periodicity of the sequence of remainders $(C_r(n)\pmod{d})_{n\geq r})$.

\begin{thm}\label{p2}
For each $r,d\in\N_+$, if $n_1,n_2\geq r$ and $n_1\equiv n_2\pmod{d}$ then
\begin{equation*}
(-1)^{n_1}C_r(n_1)\equiv (-1)^{n_2}C_r(n_2)\pmod{d}.
\end{equation*}
In particular, the sequence $(C_r(n)\pmod{d})_{n\geq r}$ is periodic of period
\begin{itemize}
\item $d$, if $2\mid d$,
\item $2d$, if $2\nmid d$.
\end{itemize}
\end{thm}

\subsection{The set of prime divisors of the sequence $C_r(n)_{n=r}^\infty$}

Now we will prove that for fixed $r\in\N_+$ there are infinitely many prime divisors of numbers $C_r(n)$, $n\in\N$. Define two subsets of the set $\bbb{P}$ of prime numbers:
\begin{equation*}
\cal{A}_r:=\{p\in\bbb{P}:\forall_{n\geq r} p\nmid C_r(n)\}, \mbox{ } \cal{B}_r:=\bbb{P}\bs\cal{A}_r.
\end{equation*}

\begin{thm}
The set $\cal{B}_r$ is infinite for any positive integer $r$.
\end{thm}

\begin{proof}
Assume that $\cal{B}_r = \{p_1, ..., p_s\}$. Let us put $d:=p_1\cdot ...\cdot p_s$.
By Theorem \ref{p2} the sequence $((-1)^nC_r(n) \pmod d)_{n\geq r}$ has period $d$.
Since $C_r(r)=1$, we have $C_r(2dm + r)\equiv\pm 1(mod d)$. Hence, $\gcd(p_i,C_r(2dm + r)) = 1$ for $i\in\{1, 2, . . . ,s\}$. Thus, either some prime other than one of $p_1, p_2, . . . , p_s$ divides $C_r(2dm + r)$ for some $m\in\N$ (an obvious contradiction), or $C_r(2dm + r) = 1$ for all $m\in\N$. On the other
hand, $\lim_{n\rightarrow +\infty}\frac{C_r(n)}{n!}
= \lim_{n\rightarrow +\infty}\frac{D_r(n)}{n!\cdot n(n-1)\cdot ...\cdot (n-r)}
= \lim_{n\rightarrow +\infty}\frac{D_r(n)}{(n+r)!} = \frac{1}{e\cdot r!}$,
by Theorem \ref{t2}. This fact implies that $C_r(n)\rightarrow +\infty$,
when $n\rightarrow +\infty$, and this is a contradiction.
\end{proof}

For a given prime number $p$ it is easy to verify whether $p\in\cal{A}_r$. Because of periodicity of the sequence $((-1)^nC_r(n) \pmod p)_{n\geq r}$ it suffices to check that $p$ divides none of the finitely many numbers $C_r(n)$, $n\in\{r,...,p+r-1\}$.

\begin{con}\label{con1}
The set $\cal{A}_r$ is infinite. Moreover, $\lim_{n\rightarrow +\infty} \frac{\sharp (\cal{A}_r\cap\{1,...,n\})}{\sharp (\bbb{P}\cap\{1,...,n\})} = \frac{1}{e}$.
\end{con}

The following heuristic reasoning allows us to claim the second
statement in the conjecture above. If we fix a prime number $p$
and randomly choose a sequence $(a_n)_{n\in\N}$ such that the
sequence of remainders $(a_n\pmod{p})_{n\in\N}$ has period $p$
then the probability that $p$ does not divide any term of this
sequence is equal to $\left(1-\frac{1}{p}\right)^p$. As $p\rightarrow +\infty$
this probability tends to $\frac{1}{e}$. Note that $p\in\cal{A}_r$ if and only if $p$ does not divide any number $C_r(n)$, $n\geq r$ and the sequence $((-1)^nC_r(n)\pmod{p})_{n\geq r}$ is periodic of period $p$. Therefore we suppose that the probability that $p\in\cal{A}_r$ tends to $\frac{1}{e}$, when $p\rightarrow +\infty$ and hence the asymptotic density of the set $\cal{A}$ in the set $\bbb{P}$ is equal to $\frac{1}{e}$.

\subsection{$p$-adic valuation}

Now we are going to study the $p$-adic valuation of
\[C_r(n)=\frac{(n-r)!}{n!}D_r(n).\]
Note that, as we explained earlier, the divisor $\frac{(n-r)!}{n!}$ is taken to cancel out the ``trivial'' divisor $\frac{n!}{(n-r)!}$.

For a given prime number $p$ we define the $p$-adic valuation of
a nonzero rational number $x$ to be an integer $t$ such that $x=\frac{a}{b}p^t$,
where $a,b\in\Z$, $b>0$ and $p\nmid ab$. We denote the $p$-adic valuation
of a number $x$ by $v_p(x)$. Moreover, we set $v_p(0)=+\infty$. In order to
describe the behavior of the $p$-adic valuation of $C_r(n)$ we give a
so-called \emph{pseudo-polynomial decomposition modulo $p$} of the sequence
$(C_r(n))_{n\geq r}$. This is a sequence of pairs $(P_{p,k},g_{p,k})_{k\geq 2}$ such that
\begin{itemize}
\item $P_{p,k}\in\Z[X]$, $g_{p,k}:\{r,r+1,r+2,...\}\rightarrow\Z\bs p\Z$, $k\geq 2$;
\item $C_r(n) \equiv P_{p,k}(n)g_{p,k}(n) \pmod{p^k}$ for all $n\geq r$, $k\geq 2$;
\item $P'_{p,k}(n) \equiv P'_{p,2}(n) \pmod{p}$ for any $k\geq 2$ and $n\geq 2$.
\end{itemize}
Note that the sequence $(\hat{f}_{r,p^k},(-1)^n)_{k\geq 2}$ is a pseudo-polynomial decomposition modulo $p$ of the sequence $(C_r(n))_{n\geq r}$. Indeed, $\hat{f}_{r,p^k}\in\Z[X]$, $p\nmid (-1)^n$ and, by (\ref{eq9}), $C_r(n)\equiv (-1)^n\hat{f}_{r,p^k}(n)\pmod{p^k}$ for any $n\geq r$ and $k\geq 2$. It remains to check that $f'_{r,p^k}(n)\equiv f'_{r,p^2}(n)\pmod{p}$, $n\geq r$, $k\geq 2$:
\begin{equation*}
\begin{split}
f'_{r,p^k}(n) & =\sum_{j=r+1}^{r+p^k-1} {j\choose r}(-1)^j\sum_{s=r}^{j-1}\prod_{i=r, i\neq s}^{j-1} (n-i)\\
& \equiv\sum_{j=r+1}^{r+p^2-1} {j\choose r}(-1)^j\sum_{s=r}^{j-1}\prod_{i=r, i\neq s}^{j-1} (n-i)=f'_{r,p^2}(n)\pmod{p}.
\end{split}
\end{equation*}
If $j\geq r+p^2$ then for each $n\in\N$ in the product $\prod_{i=r}^{j-1} (n-i)$ there appear at least two factors divisible by $p$. Hence $p\mid\prod_{i=r, i\neq s}^{j-1} (n-i)$ for any $s\in\{r, ..., j-1\}$ and via reduction modulo $p$ we can skip the summand ${j\choose r}(-1)^j\sum_{s=r}^{j-1}\prod_{i=r, i\neq s}^{j-1} (n-i)$ for $j\geq r+p^2$.

In the followings we need a result which describes the $p$-adic valuation of a sequence $(a_n)_{n\in\N}$ with pseudo-polynomial decomposition modulo $p$ (see \cite[Theorem 1, p. 4]{Mi}).

\begin{thm}[Hensel's lemma for pseudo-polynomial decomposition modulo $p$]\label{t3}
Let $p$ be a prime number, $k\in\N_+, n_k\in\N$ be such that $p^k\mid a_{n_k}$ and assume that $(a_n)_{n\in\N}\subset\Z$ has a pseudo-polynomial decomposition modulo $p$. Let us denote this decomposition by $(P_{p,k},g_{p,k})_{k\geq 2}$ and define
\[q_p(n_k)=\frac{1}{p}\left(\frac{a_{n_k+p}}{g_{p,2}(n_k+p)}-\frac{a_{n_k}}{g_{p,2}(n_k)}\right).\]
\begin{itemize}
\item If $v_p(q_p(n_k))=0$ then there exists a unique $n_{k+1}$ modulo $p^{k+1}$ for which $n_{k+1}\equiv n_k \pmod{p^k}$ and $p^{k+1}\mid{a_n}$ for all $n\equiv n_{k+1}\pmod{p^{k+1}}$. What is more, $n_{k+1} \equiv n_k-\frac{a_{n_k}}{g_{p,k+1}(n_k)q_p(n_k)} \pmod{p^{k+1}}$.
\item If $v_p(q_p(n_k))>0$ and $p^{k+1}\mid{a_{n_k}}$ then $p^{k+1}\mid{a_n}$ for all $n\equiv n_k \pmod{p^k}$.
\item If $v_p(q_p(n_k))>0$ and $p^{k+1}\nmid{a_{n_k}}$ then $p^{k+1}\nmid{a_n}$ for any $n\equiv n_k \pmod{p^k}$.
\end{itemize}
In particular, if $k=1$, $p\mid{a_{n_1}}$ and $v_p(q_p(n_1))=0$ then for any $l\in\mathbb{N}_+$ there exists a unique $n_l$ modulo $p^l$ such that $n_l\equiv n_1 \pmod{p}$ and $v_p(a_n)\geq{l}$ for all $n\equiv n_l\pmod{p^l}$. Moreover, $n_l$ satisfies the congruence $n_l \equiv n_{l-1}-\frac{a_{n_{l-1}}}{g_{p,l}(n_{l-1})q_p(n_1)} \pmod{p^l}$ for $l>1$.
\end{thm}

Using Theorem \ref{t3} we obtain the below description of the $p$-adic valuation of $C_r(n)$.

\begin{cor}
Let $p$ be a prime number, and $k\in\N_+, n_k\geq r$ be such that $p^k\mid C_r(n_k)$. Let $\hat{q}_p(n_k)=\frac{1}{p}\left((-1)^pC_r(n_k+p)-C_r(n_k)\right)$.
\begin{itemize}
\item If $v_p(\hat{q}_p(n_k))=0$ then there exists a unique $n_{k+1}$ modulo $p^{k+1}$ for which $n_{k+1}\equiv n_k \pmod{p^k}$ and $p^{k+1}\mid{C_r(n)}$ for all $n\geq r$ such that $n\equiv n_{k+1}\pmod{p^{k+1}}$. What is more, $n_{k+1} \equiv n_k-\frac{C_r(n_k)}{\hat{q}_p(n_k)} \pmod{p^{k+1}}$.
\item If $v_p(\hat{q}_p(n_k))>0$ and $p^{k+1}\mid{C_r(n_k)}$ then $p^{k+1}\mid{C_r(n)}$ for all $n\geq r$ such that $n\equiv n_k \pmod{p^k}$.
\item If $v_p(\hat{q}_p(n_k))>0$ and $p^{k+1}\nmid{C_r(n_k)}$ then $p^{k+1}\nmid{C_r(n)}$ for any $n\geq r$ such that $n\equiv n_k \pmod{p^k}$.
\end{itemize}
In particular, if $k=1$, $p\mid{a_{n_1}}$ and $v_p(\hat{q}_p(n_1))=0$ then for any $l\in\mathbb{N}_+$ there exists a unique $n_l$ modulo $p^l$ such that $n_l\equiv n_1 \pmod{p}$ and $v_p(C_r(n))\geq{l}$ for all $n\geq r$ such that $n\equiv n_l\pmod{p^l}$. Moreover, $n_l$ satisfies the congruence $n_l \equiv n_{l-1}-\frac{C_r(n_{l-1})}{(-1)^{n_1+n_{l-1}}\hat{q}_p(n_1)} \pmod{p^l}$ for $l>1$.
\end{cor}

\begin{proof}
It suffices to apply Theorem \ref{t3} for the sequence $(a_n)_{n\in\N}=(C_r(n+r))_{n\in\N}$ after noting that $\hat{q}_p(n_k)=(-1)^{n_k}q_p(n_k)$, where $q_p(n_k)$ is as specified in Theorem \ref{t3}.
\end{proof}

\section{Some diophantine equations involving $r$-derangements numbers}

In \cite{Mi} it was shown that there are only finitely many solutions of the diophantine equation $D(n)=q\cdot m!$, where $q\in\Q$ is fixed and $n,m\in\N_+$ are unknown variables. The knowledge on the $2$-adic valuation of the derangement number $D(n)$ plays a crucial role in the proof of this result. We can provide analogous reasoning for the $r$-derangements to obtain similar finiteness result for the diophantine equation $D_r(n)=q\cdot m!$, where $r\geq 2$ and the set $\cal{A}_r$ is nonempty.

\begin{thm}\label{p3}
Assume that $r\geq 2$ is fixed and $\cal{A}_r\neq\emptyset$. Then for any nonzero rational parameter $q$ the diophantine equation $D_r(n)=q\cdot m!$ with unknowns $n,m\in\N_+$ has only finitely many solutions. All the solutions $(n,m)$ satisfy one of the following equalities:
\begin{equation*}
\begin{split}
m< & (p-1)\left(1+\log_p m+r\log_p\left(r+er(2+q)\right)-v_p(q)\right), \\
m> & p^{\frac{v_p(q)}{r+1}}\cdot p^{\frac{m}{(r+1)(p-1)}},
\end{split}
\end{equation*} where $p$ is a prime number from the set $\cal{A}_r$.
\end{thm}

\begin{proof}
It suffices to prove the statement for $q>0$ because the equation $D_r(n)=q\cdot m!$ has no solutions for $q<0$.

Assume that $D_r(n)=q\cdot m!$. Then $n\geq r$ and $v_p(D_r(n))=v_p\left(\frac{n!}{(n-r)!}\right)=v_p(n(n-1)...(n-r+1))$, since $D_r(n)=\frac{n!}{(n-r)!}C_r(n)$ and $p\nmid C_r(n)$. We thus obtain the following chain of inequalities:
\begin{equation*}
n^r>n(n-1)...(n-r+1)=p^{v_p\left(n(n-1)...(n-r+1)\right)}=p^{v_p(D_r(n))}=p^{v_p(q\cdot m!)}.
\end{equation*}
Hence $n\geq\lceil p^{\frac{v_p(q\cdot m!)}{r}}\rceil$. Let us put $M(m)=\lceil p^{\frac{v_p(q\cdot m!)}{r}}\rceil$. The sequence $(D_r(n))_{n\geq r}$ is increasing, which means that $q\cdot m!=D_r(n)\geq D_r(M(m))$. By Theorem \ref{t2} we know that
\[D(M(m))>\frac{M(m)!}{e}{M(m)-1\choose r}-2M(m)!{M(m)-1\choose r-1}
= M(m)!{M(m)-1\choose r-1}\left(\frac{M(m)-r}{er}-2\right).\]
We thus conclude that
\begin{equation*}
q\cdot m!>M(m)!{M(m)-1\choose r-1}\left(\frac{M(m)-r}{er}-2\right),
\end{equation*}
or, equivalently,
\begin{equation}\label{ineq1}
m!>M(m)!{M(m)-1\choose r-1}\frac{1}{q}\left(\frac{M(m)-r}{er}-2\right).
\end{equation}
However, it is easy to check that if the system of inequalities
\begin{align}
M(m)&\geq m\label{ineq2}\\
M(m)&\geq r\label{ineq3}\\
\frac{1}{q}\left(\frac{M(m)-r}{er}-2\right)&\geq 1\label{ineq4}
\end{align}
is satisfied then (\ref{ineq1}) does not hold (indeed, (\ref{ineq2})
implies that $M(m)!\geq m!$ and (\ref{ineq3}) implies that
${M(m)-1\choose r-1}\geq 1$). First we will investigate
\begin{equation}\label{ineq5}
m\geq (p-1)\left(1+\log_p m + r\log_p\left(r+er(2+q)\right)-v_p(q)\right)
\end{equation}
implies (\ref{ineq2}) and (\ref{ineq3}). We have the following chain of conclusions:
\[m\geq (p-1)\left(1+\log_p m + r\log_p\left(r+er(2+q)\right)-v_p(q)\right)\]
\[\Leftrightarrow \frac{v_p(q)+\frac{m}{p-1}-\log_p m -1}{r}\geq\log_p\left(r+er(2+q)\right)\Leftrightarrow  p^{\frac{v_p(q)+\frac{m}{p-1}-\log_p m -1}{r}}\geq r+er(2+q).
\]
By the definition of $M(m)$ and Legendre's formula $v_p(m!)=\frac{m-s_p(m)}{p-1}$, where $s_p(m)$ is the sum of $p$-ary digits of the number $m$ (see \cite{Le}), we get the inequalities
\begin{equation}\label{ineq6}
M(m)\geq p^{\frac{v_p(q\cdot m!)}{r}}= p^{\frac{v_p(q)+\frac{m-s_p(m)}{p-1}}{r}}> p^{\frac{v_p(q)+\frac{m}{p-1}-1-\log_p m}{r}}.
\end{equation}
That is, $M(m)>r+er(2+q)$, which implies (\ref{ineq3}) and (\ref{ineq4}). The inequality
\begin{equation}\label{ineq7}
m\leq p^{\frac{v_p(q)}{r+1}}\cdot p^{\frac{m}{(r+1)(p-1)}}
\end{equation}
after raising to the power of $\frac{r+1}{r}$ and dividing by $m^{\frac{1}{r}}$ gives that
\begin{equation*}
m\leq p^{\frac{v_p(q)+\frac{m}{p-1}-1-\log_p m}{r}},
\end{equation*}
which, combined with (\ref{ineq6}), implies (\ref{ineq2}).

Summarizing our results, we state that if $m$ satisfies (\ref{ineq5}) and (\ref{ineq7})
then $D_r(n)\neq q\cdot m!$ for any $n\in\N_+$. Moreover, the right-hand side
of (\ref{ineq5}) is a function of the variable $m$ with derivative decreasing
to $0$ as $m\to+\infty$. Similarly, the right-hand side of (\ref{ineq7}) is a
function of $m$ with derivative increasing to $+\infty$ as $m\to+\infty$. Thus
there exists an $m_0\in\N_+$ such that each $m\geq m_0$ satisfies (\ref{ineq5})
and (\ref{ineq7}). Hence $q\cdot m!$ cannot be an $r$-derangement number for any
$n\in\N_+$ if only $m\geq m_0$.
\end{proof}

\begin{ex}
{\rm We will use Theorem \ref{p3} to determine all the $r$-derangements numbers which are factorials, at least for $r\in\{2,3\}$.

If $r=2$ then the numbers $C_2(2)=1$, $C_2(3)=2$ and $C_3(4)=7$ are not divisible by $3$. Hence $3\in\cal{A}_2$. By Theorem \ref{p3}, if $D_2(n)=m!$ then $m<2\left(1+\log_3 m+2\log_3\left(2+6e\right)\right)<13+2\log_3 m$ or $m>3^{\frac{m}{6}}$. This means that $m\leq 18$ and checking $m\in\{1,...,18\}$ one by one we claim that the only solution $(n,m)$ of the diophantine equation $D_2(n)=m!$ is $(2,2)$.

If $r=3$ then the numbers $C_3(3)=1$ and $C_3(4)=3$ are odd. Hence $2\in\cal{A}_3$. By Theorem \ref{p3}, if $D_3(n)=m!$ then $m<1+\log_2 m+3\log_2\left(3+9e\right)<6+\log_2 m$ or $m>2^{\frac{m}{4}}$. This means that $m\leq 15$ and checking $m\in\{1,...,15\}$ one by one we claim that the only solution $(n,m)$ of the diophantine equation $D_3(n)=m!$ is $(3,3)$.}
\end{ex}

It is worth to note that for $r\geq 3$ there are no powers of prime numbers in the sequence of $r$-derangements numbers, since $r!\mid D_r(n)$ and $r!$ is not a power of prime number. If $r=1$ then $D_1(n)=D(n+1)$ for each $n\in\N$ and the problem of diophantine equation $D(n)=p^k$ with unknowns $n,k\in\N_+$ and $p\in\bbb{P}$ ($\bbb{P}$ denotes the set of all prime numbers) was considered in \cite{Mi}. In case $r=2$ the number $D_2(2)=2$ is the only power of a prime number in the sequence of $2$-derangements.

\begin{thm}
The only solution of the diophantine equation $D_2(n)=p^k$ with unknowns $p\in\bbb{P}$ and $n,k\in\N_+$ is $(p,n,k)=(2,2,1)$.
\end{thm}

\begin{proof}
We know that $D_2(0)=D_2(1)=0$, $D_2(2)=2$ and $D_2(3)=2\cdot 3!=12$. Let us assume now that $n\geq 4$. Then $\frac{n(n-1)}{2}\mid D_2(n)$ and $D_2(n)>0$. If $2\mid n$ then $\frac{n}{2}$ and $n-1$ are coprime integers greater than $1$, thus $\frac{n(n-1)}{2}$ is not a power of a prime number. If $2\nmid n$ then $\frac{n-1}{2}$ and $n$ are coprime integers greater than $1$ so, again, $\frac{n(n-1)}{2}$ is not a power of a prime number.
\end{proof}

\section*{Acknowledgements}

A part of this paper emerged from a discussion between P. Miska and I. Mez\H{o} on the \emph{Journ\'ees Arithm\'etiques 2015} conference in Debrecen, Hungary.

P. Miska wishes to thank his advisor, Maciej Ulas, for suggestion of considering polynomials $P_n(X)$ from Section \ref{polyrel}.

\end{document}